\documentclass[usletter,11pt,reqno]{amsart}

\usepackage{amssymb,amscd,amsthm, verbatim,amsmath,fancyhdr, mathrsfs}
\usepackage{graphicx}
\usepackage{turnstile}
\usepackage{url}
\usepackage[hidelinks]{hyperref}
\usepackage{amsaddr}
\usepackage{pxfonts}

\usepackage{caption} 
\usepackage[margin=.98in]{geometry}

\def\lpf{\mathrm{lpf}}
\def\spf{\mathrm{spf}}
\def\p{\mathfrak{p}}

\newtheorem{thm}{Theorem}[section]

\newtheorem{cor}[thm]{Corollary}

\title[Linear Diophantine equations with the sum of divisors]{Computing bounded solutions to linear Diophantine equations with the sum of divisors}

\author{Max A. Alekseyev}
\address{The George Washington University\\Washington, DC, USA}
\email{maxal@gwu.edu}

\date{}

\begin{document}

\begin{abstract}
We propose an efficient computational method for finding all solutions $n\leq U$ to the Diophantine equation $a\sigma(n)=bn+c$, where $a,b,c$ are fixed integer coefficients, $U$ is a given upper bound, and $\sigma(n)$ denotes the sum of divisors of $n$.
The method combines advanced pruning of possible prime factors of $n$ and their combinations with several optimization techniques, organized around a \emph{prime-wheel} data structure developed for this purpose.
It is implemented in the {\sc SageMath} computer algebra system within the framework of recursively enumerated sets and supports native {\sc MapReduce} parallelization. We apply the method to numerous previously studied equations, discovering new solutions, certifying complete solution lists below substantially larger bounds than previously reported, and substantially improving lower bounds in several open problems involving the sum-of-divisors function.
\end{abstract}

\maketitle

\section{Introduction}

The \emph{sum of divisors} function, commonly denoted by $\sigma$, has been studied for centuries. In particular, it provides elegant characterizations of several important classes of integers. The \emph{prime numbers} are precisely the solutions to $\sigma(n)=n+1$, while the \emph{perfect numbers} are defined by $\sigma(n)=2n$; several related classes are discussed later in the paper. While the solutions to the former equation are completely understood, the latter is only partially understood, since the existence of an odd perfect number is one of the oldest open questions in number theory. This question is representative of the rich collection of unresolved problems concerning equations involving $\sigma$~\cite[Section B2]{Guy2004}.

The aforementioned equations can be viewed as special cases of the Diophantine equation:
\begin{equation}\label{eq:main}
a\sigma(n) = bn + c,
\end{equation}
where $a>0,b,c$ are fixed integer coefficients with $\gcd(a,b,c)=1$ and $n$ is a positive integer unknown.
In the present study, we develop an efficient computational method for finding all solutions to a given equation \eqref{eq:main} up to a given bound $U$. 

Note that the case $a=1$ and $c=0$ corresponds to \emph{multiperfect numbers}, more specifically \emph{$b$-perfect numbers} or just perfect numbers when $b=2$. This case has been the subject of extensive theoretical study (see, for example, references in~\cite[Section~B1]{Guy2004}) as well as large-scale computational searches~\cite{Flammenkamp2023}. Although we do not exclude the case $c=0$ from consideration, it is rather special as it admits additional optimization techniques that are not available for nonzero values of $c$. Another special case $b=0$ corresponds to inverting the sum of divisors function, a problem
addressed in~\cite{Alekseyev2016}; we therefore handle this case with the corresponding software.
Accordingly, in the present paper we focus on the general case and do not discuss specialized treatments for $b=0$ or $c=0$.

We apply our method to many equations of the form~\eqref{eq:main}, particularly those that are present in the Online Encyclopedia of Integer Sequences (OEIS)~\cite{OEIS}, and extend knowledge of their ``small'' solutions by discovering new terms and certifying complete solution lists below substantially larger search bounds than previously reported. Similarly, for equations with no known solutions (such as quasiperfect and odd almost-perfect numbers~\cite[Section B2]{Guy2004}), our method can significantly raise the known lower bounds for potential solutions. 

The paper is organized as follows. We introduce the notation in Section~\ref{sec:notation}, describe the proposed method in Section~\ref{sec:outline} and its implementation in Section~\ref{sec:implement}, and then present some practical results in Section~\ref{sec:applications}. We analyze the search space and complexity in Section~\ref{sec:complexity}. We conclude the paper with a discussion in Section~\ref{sec:remarks}.

\section{Notation}\label{sec:notation}

We start by introducing the following notation, which we use throughout the paper:
\begin{itemize}
\item $\spf(n)$ and $\lpf(n)$ denote the smallest and largest prime factors of an integer $n>1$, respectively; 
\item $\nu_p(n)$ denotes the $p$-adic valuation of $n$, that is, the largest exponent $k$ such that $p^k\mid n$;
\item $\Omega(n)$ and $\omega(n)$ denote the number of prime factors of $n$ with and without multiplicities, respectively;
\item $\tau(n)$ denotes the number of divisors of an integer $n$;
\item $\p_1=2$, $\p_2=3$, $\p_3=5$, $\dots$ denote the prime numbers in their natural order.
\end{itemize}

\section{Method outline}\label{sec:outline}

At its core, our approach to solving \eqref{eq:main} for $n\leq U$ is based on representing the positive integers not exceeding $U$ as the nodes of a tree $T_U$ rooted at $1$, where each node $n>1$ has the parent $n/p^{\nu_p(n)}$ with $p:=\lpf(n)$ (Fig.~\ref{fig:treeTU}). We perform a (constrained) depth-first traversal of $T_U$ with a few important optimization techniques, which we describe in the following subsections. We therefore refer to the nodes of $T_U$ as the \emph{search space} and to $U$ as the \emph{search bound}.

\begin{figure}[!tb]
\centering \includegraphics[width=\textwidth]{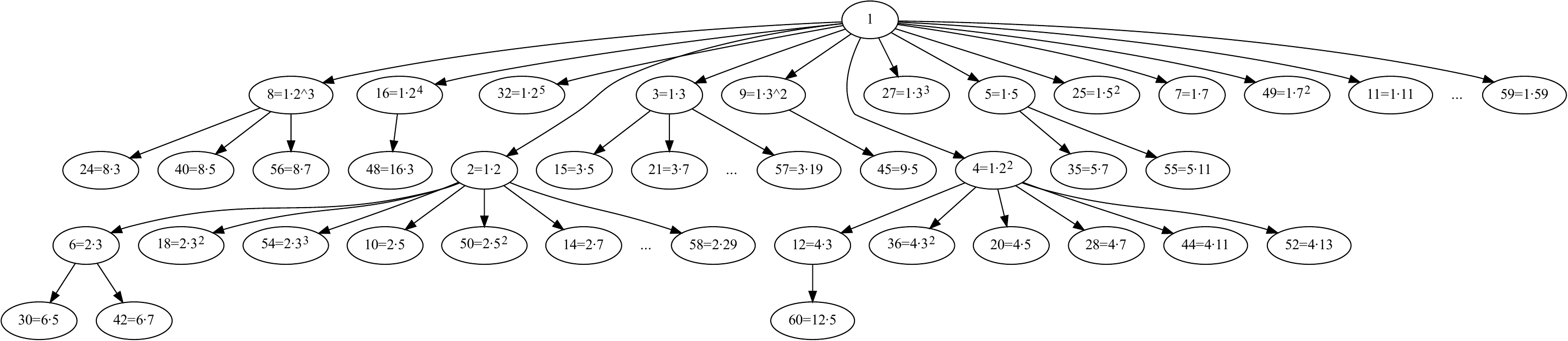}
\caption{The tree $T_U$ for $U=60$, where some nodes $1\cdot p$, $2\cdot p$, and $3\cdot p$ with prime $p$ are hidden under ellipses.}
\label{fig:treeTU}
\end{figure}


Note that the descendants of a node $m$ have the form $mn'$, where $\spf(n') > \lpf(m)$ (thus $\gcd(m,n')=1$) and $n'\leq U':=U/m$ satisfies the equation $a'\sigma(n') = b'n' + c'$ with the coefficients $(a',b',c')$ obtained from $(a\sigma(m), bm, c)$ by canceling their common factor (see also Section~\ref{sec:reduceabc}). 
That is, at node $m$, we are essentially solving the equation \eqref{eq:main} with $(a,b,c,U)=(a',b',c',U')$ for $n=n'$ under an additional constraint $\spf(n') > \lpf(m)$.

\subsection{Shortcuts}\label{sec:shortcuts}

By a \emph{shortcut} we mean a way to determine the solutions $n=mn'$ with $\Omega(n') \leq 2$ or $\omega(n')=1$ among the descendants of a node $m$ in $T_U$, without traversing all those descendants.
There are two cases to consider.

\

{\bf Case $n'=p^k$} with $k\geq 1$ and prime $p\in (\lpf(m), U'^{1/k}]$. The equation \eqref{eq:main} here takes the form $a'\frac{p^{k+1}-1}{p-1} = b'p^k + c'$, implying that prime $p$ divides $a' - c'$.
If $a'\ne c'$, we factor $a'-c'$ and try its prime factors as candidate values for $p$, for each of which we then determine suitable values of the exponent $k$.
Otherwise, if $a'=c'$, then $\gcd(a',b')=1$ and $((a'-b')p+b')p^{k-1} = a'$, implying that
\begin{itemize}
    \item for $k=1$ and $a' \ne b'$, there are no solutions; 
    \item for $k=1$ and $a' = b'$ (thus $a'=b'=c'$), any prime $p>\lpf(m)$ gives a solution $mp$ (in our implementation, the case $a'=b'=c'$ does not arise because it is handled in preprocessing; see Section~\ref{sec:reduceabc});
    \item for $k\geq 2$, $p^{k-1}\mid a'$ and, moreover, $k-1=\nu_p(a')$, that is, the candidate values for $(p,k)$ are derived from the prime power factors of $a'$.
\end{itemize}
Among the identified solutions we may or may not discard those with $p > U'^{1/k}$, an option we discuss in Section~\ref{sec:aboveU}.

A concrete special case of this shortcut for $(a,b,c)=(1,2,d)$ is given by the following elementary observation, independently noted by several contributors to the corresponding OEIS entries (see Section~\ref{sec:small_abundance}):
\begin{thm}[OEIS~\cite{OEIS}]\label{th:abundance}
 For integers $d$ and $\ell>0$, an integer $n=2^{\ell-1}(2^\ell-d-1)$ is a solution to $\sigma(n)=2n+d$ whenever $2^\ell - d - 1$ is prime.
\end{thm}
Indeed, here we have $m=2^{\ell-1}$, giving $(a',b',c')=(\sigma(m),2m,d)=(2^\ell-1,2^\ell,d)$, and if $p:=a'-c'=2^\ell-1-d$ is prime, then the shortcut produces exactly the solution $mp = 2^{\ell-1}(2^\ell-d-1)$ stated in Theorem~\ref{th:abundance}.

\

{\bf Case $n'=pq$} with distinct primes $p, q$, both greater than $\lpf(m)$, and $pq\leq U'$. Here the equation \eqref{eq:main} takes the form $a'(p+1)(q+1) = b'pq + c'$, which we rewrite as $Apq + Bp + Bq + C = 0$ with coefficients $A:=a'-b'$, $B:=a'$, and $C:=a'-c'$ (in practice, we also cancel their common factor to have $\gcd(A,B,C)=1$). 
If $A=0$, we check if $B\mid C$, in which case we obtain the suitable pair $(p,q)$ by iterating $p$ over the primes in the interval $\left(\lpf(m),\min(D/2,\sqrt{U'})\right)$ with $D:=-C/B$, and testing $q:=D-p$ for primality; if $B\nmid C$, then there are no solutions.
Otherwise, when $A\ne 0$, we \emph{complete the rectangle}, a technique already known to Brahmagupta~\cite[Chapter XIII]{Dickson1920}, and rewrite the equation as
$$(Ap+B)(Aq+B)=B^2-AC,$$
which allows us to efficiently obtain the suitable prime pairs $(p,q)$ by factoring and iterating over the divisors of $B^2-AC$. In the exceptional case $B^2-AC=0$, solutions exist only if $p:=-B/A$ is a prime satisfying $p>\lpf(m)$, and then we report that any prime $q>\lpf(m)$ different from $p$ gives a solution $mpq$. The following theorem lets us rule out some cases without factoring.

\begin{thm}\label{th:fact_bounds} 
Suppose that real numbers $x,y$ satisfy the inequalities $x\geq L$, $y\geq L$, and $xy\leq K$ for bounds $L,K$ with $L^2\leq K$. Then for any coefficients $A>0$ and $B$, we have
$$
\max_{x,y}\ (Ax+B)(Ay+B) = 
\begin{cases}
(AL+B)(A\frac{K}L+B), &\text{if } B\geq 0;\\
(A\sqrt{K}+B)^2, &\text{if } -A\frac{\sqrt{K}+L}2 \leq B < 0;\\
(AL+B)^2, &\text{if } B< -A\frac{\sqrt{K}+L}2.
\end{cases}
$$
and
$$
\min_{x,y}\ (Ax+B)(Ay+B) = 
\begin{cases}
(AL+B)(A\frac{K}L+B), &\text{if } B< -AL;\\
(AL+B)^2, &\text{if } B\geq -AL.
\end{cases}
$$
\end{thm}

\begin{proof}
Note that $(Ax+B)(Ay+B) = A^2xy + AB(x+y) + B^2$. 
When $B\geq 0$, the minimum and maximum of $(Ax+B)(Ay+B)$ are attained at $(x,y)=(L,L)$ and $(x,y)=(L,\frac{K}L)$, respectively, where both $x+y$ and $xy$ are simultaneously minimized and maximized. For the remainder, assume that $B<0$.

For any fixed value $s:=x+y\in [2L, L+\frac{K}L]$, we have $L(s-L)\leq xy\leq s^2/4$. The minimum of $(Ax+B)(Ay+B)$ equals the minimum of $g(s) := A^2 L(s-L) + ABs + B^2 = A(AL+B)s -A^2L^2 + B^2$, which is given by $g(2L)$ if $AL+B\geq 0$, and $g(L+\frac{K}L)$ otherwise.

Similarly, the maximum of $(Ax+B)(Ay+B)$ equals the maximum of $f(s) := A^2\min\{s^2/4,K\} + ABs + B^2$, a piecewise quadratic-linear function with a corner at $s=2\sqrt{K}$. It follows that the maximum of $f(s)$ is given by $\max\{ f(2L), f(2\sqrt{K})\}$. The two expressions are equal when $B = -A\frac{\sqrt{K}+L}2$, leading to the stated formula for $\max_{x,y}\ (Ax+B)(Ay+B)$.
\end{proof}

Theorem~\ref{th:fact_bounds} for $L=\lpf(m)$ and $K=U'$ gives lower and upper bounds for $(Ap+B)(Aq+B)$, and lets us rule out suitable prime pairs $(p,q)$ without factoring $B^2-AC$ when that value lies outside those bounds. In our implementation, this check can be turned off (see Section~\ref{sec:aboveU}) to allow the possibility of obtaining some solutions above $U$ at the cost of extra factoring work.


\subsection{Pruning with prime wheel}

At each node $m$ of $T_U$, the shortcuts described in the previous section produce all solutions $n=mn'$ satisfying $\Omega(n') \leq 2$ or $\omega(n')=1$, and so it remains to focus on finding those with $\Omega(n')\geq 3$ and $\omega(n')\geq 2$. It immediately follows that $\spf(n')\leq U'^{1/3}$; however, we do not need this bound as our approach relies on more accurate dynamic bounds as described below.

Our goal is to generate a set $Q$ containing all \emph{feasible} prime powers, that is, for any solution $n'$ of interest and $p:=\spf(n')$, we require $p^{\nu_p(n')}\in Q$. Since $Q$ determines the set of children $\{ mq\ :\ q\in Q\}$ of the node $m$ to visit, we want $Q$ to be as small as possible.
We use the following theorem, which may be viewed as a refinement of Lemma~1 in \cite{Hagis1982} and Lemma~2(iii) in~\cite{Anavi2013}:

\begin{thm}\label{th:bound} Let $n$, $U$, and $S$ be positive integers such that $n\leq U$, $\sigma(n)\geq S$, and $\spf(n)=\p_k$ for some index $k$. Then, for every positive integer $\ell$:
\begin{itemize}
    \item if $\ell \leq \omega(n)$, then 
    $$\prod_{i=1}^{\ell} \p_{k+i-1} \leq U;$$
    \item if $\ell \geq \omega(n)$, then 
    $$\prod_{i=1}^{\ell} \frac{\p_{k+i-1}}{\p_{k+i-1}-1} \geq \frac{S}{U}.$$
\end{itemize}
\end{thm}

\begin{proof} First consider the case $\ell=\omega(n)$.
Since $\spf(n)=\p_k$, the $\ell$ distinct prime factors of $n$ in increasing order are bounded from below by $\p_k,\p_{k+1},\dots,\p_{k+\ell-1}$, respectively, and therefore
\begin{equation}\label{eq:boundP0}
  \prod_{i=1}^{\ell} \p_{k+i-1} \leq n \leq U.  
\end{equation}
For the fraction $\frac{\sigma(n)}n$, we have the following upper bound:
$$\frac{\sigma(n)}n = \prod_{\text{prime }p|n} (1+p^{-1}+\dots+p^{-\nu_p(n)})\leq \prod_{\text{prime }p|n} \frac{p}{p-1}.$$
Since $\frac{p}{p-1}$ is a decreasing function of $p$, the following inequality holds:
\begin{equation}\label{eq:boundP1}
\prod_{i=1}^{\ell} \frac{\p_{k+i-1}}{\p_{k+i-1}-1} \geq \frac{\sigma(n)}{n} \geq \frac{S}{U}.
\end{equation}
The general statement now follows from the observation that for a fixed $k$, the left-hand sides of the inequalities \eqref{eq:boundP0} and \eqref{eq:boundP1} represent increasing functions of $\ell$.
\end{proof}

We construct the set $Q$ by keeping track of a dynamic lower bound $\ell\leq \omega(n')$ (initially $\ell=2$) and an $\ell$-tuple of consecutive primes $W:=(\p_k,\p_{k+1},\dots,\p_{k+\ell-1})$, starting with $W_1 = \p_k$ (initially $\p_k$ is the next prime after $\lpf(m)$).\footnote{We do not track the actual value of index $k$, and we use indices only to indicate the relationship between primes in $W$.}
We refer to $W$ as the \emph{prime wheel} of length $|W|=\ell$. It supports two operations:\footnote{In practice, both operations on the wheel are done by using a single generator of consecutive primes.}
\begin{description}
    \item[rolling] corresponds to incrementing $k$, when the tuple $W$ changes by removing the first element and appending the next prime ($=\p_{k+\ell}$) after the last element of $W$;
    \item[length increment] is done by appending the next prime ($=\p_{k+\ell}$) after the last element of $W$.
\end{description}


Along with the wheel $W$, we keep track of the products 
$$P_\kappa(W):=\prod_{p\in W} (p-\kappa),\qquad \kappa \in \{0,1\}.$$ 
From $|W|\leq \omega(n')$ and $W_1\leq \spf(n')$, it follows that $P_0(W) \leq n' \leq U'$, and thus $a'\frac{\sigma(n')}{n'}=b' + \frac{c'}{n'}$ is bounded from below by
\begin{equation}\label{eq:LW}
L(W) := \begin{cases}    
b' + \frac{c'}{U'} & \text{if } c'\geq 0;\\
b' + \frac{c'}{P_0(W)} & \text{if } c' < 0.
\end{cases}
\end{equation}
For each state of the wheel $W$, we test the following conditions:
\begin{itemize}
    \item if $P_0(W) > U'$, then by Theorem~\ref{th:bound} no solutions with $\spf(n')\geq W_1$ exist, and we stop the wheel;
    \item if $a'\frac{P_0(W)}{P_1(W)} < L(W)$, then by Theorem~\ref{th:bound} there are no solutions with $\omega(n')=|W|$, and we increment the wheel length.
\end{itemize}
If neither condition holds, we consider $p:=W_1$ as a candidate for $\spf(n')$.
Since $\omega(n')\geq |W|$, the power $p^t$ in $n'$ must satisfy the inequality $p^t\frac{P_0(W)}{p}\leq U'$, and so we add to $Q$ the powers $p^t$ for integers $t$ in the interval $[1, 1+\lfloor \log_p \frac{U'}{P_0(W)}\rfloor]$.
Then we continue with rolling the wheel.

Since $P_0(W)$ grows as the wheel $W$ rolls or increases in length, the wheel eventually stops. When it does, the set $Q$ contains all feasible prime powers, as shown by the following theorem:

\begin{thm}\label{th:prime_wheel}
Let $a',b',c',U'$ be defined as above.
Suppose $n'\leq U'$ is a solution to $a'\sigma(n')=b'n'+c'$ with $\omega(n')\geq 2$ and $\spf(n')=\p_t>\lpf(m)$ for some index $t$.
Then the prime wheel eventually reaches a state 
with $|W|\leq \omega(n')$ and $W_1=\p_t$.
\end{thm}

\begin{proof}
The wheel $W$ starts with length $|W|=2$ and with $W_1$ equal to the next prime after $\lpf(m)$. Hence, at the beginning we have $W_1\leq \p_t$ and $|W|\leq \omega(n')$. 
Let $W':=(\p_t,\p_{t+1},\dots,\p_{t+\omega(n')-1})$. 
We have:
\begin{itemize}
    \item if $|W|\leq \omega(n')$, then
    $$
    P_0(W) \leq P_0(W') \leq n' \leq U';    
    $$
    \item if $|W| = \omega(n')$, then again $P_0(W) \leq P_0(W')$, which together with Theorem~\ref{th:bound} further implies
    $$
    L(W) \leq L(W') \leq a'\frac{\sigma(n')}{n'} \leq a'\frac{P_0(W')}{P_1(W')} \leq a'\frac{P_0(W)}{P_1(W)}.
    $$
\end{itemize}
It follows inductively along the successive values of $W_1$ that while $W_1\leq \p_t$, the wheel $W$ does not stop (since $P_0(W) \leq U'$) and cannot grow in length above $\omega(n')$ (since $L(W) \leq a'\frac{P_0(W)}{P_1(W)}$).
That is, eventually $W$ reaches the state with $|W|\leq \omega(n')$ and $W_1=\p_t$.
\end{proof}

For simplicity, the preceding wheel description and analysis omitted the lower bound for $\Omega(n')$. In fact, knowing that $\Omega(n') \geq \ell_\Omega$ for some $\ell_\Omega\geq 3$ yields the sharper lower bound for $n'$, which is $n'\geq W_1^{\ell_\Omega-|W|} P_0(W)$ instead of just $P_0(W)$, and thus we replace $P_0(W)$ with $W_1^{\ell_\Omega-|W|} P_0(W)$ in the wheel exit condition and in the definition 
\eqref{eq:LW} for $L(W)$.

In the implementation, $\ell_\Omega$ is a dynamic bound attached to the current residual problem, rather than a fixed global parameter. Successor generation starts with $\omega(n')\geq2$ and $\Omega(n')\geq3$, because the complementary cases are handled by the shortcuts. Bounds supplied through the optional parameters \texttt{omega} and \texttt{bigomega} are inherited by descendants, and they may be strengthened by other information: for example, a prescribed value of $\tau(n')$ constrains both $\omega(n')$ and $\Omega(n')$, while an odd-square residual forces even exponents. When a child fixes a factor $p^t$, the residual lower bounds on $\Omega$ are decreased by $t$. Likewise, if the wheel has to grow, the lower bound for $\omega(n')$ grows with it and the $\Omega$ lower bound is raised accordingly. Thus the factor $W_1^{\ell_\Omega-|W|}$ accounts for multiplicities that are still required but are not represented by the distinct primes currently in $W$.

\subsection{Case of odd $\sigma(n')$}\label{sec:odd_sigma}

Our method detects when both $a'$ and $b'+c'$ are odd. In this case, for any \emph{odd} solution $n'$, we have $$\sigma(n')\equiv a'\sigma(n') = b'n' + c'\equiv b'+c'\equiv 1\pmod2,$$
implying that $n'$ is an odd square. We take this observation into account by adjusting the pruning and construction of the set $Q$ described above. In particular, when $p:=\spf(n')>2$ and hence $n'$ is an odd square, the wheel stop condition $W_1^{\ell_\Omega-|W|} P_0(W) > U'$ changes to $W_1^{\ell_\Omega-2|W|} P_0(W)^2 > U'$, and we restrict our attention only to even exponents $t$ while adding powers $p^t$ to $Q$. Additionally, from $a'\sigma(p^t)\sigma(n'/p^t)=b'n'+c'$, it follows that for any prime $q\mid \sigma(p^t)$, $-b'c' \equiv (b')^2n' \pmod{q}$, that is, $-b'c'$ is a quadratic residue modulo $q$. We test this condition by computing the Legendre symbol $\left(\frac{-b'c'}{q}\right)$ and discard $t$ whenever it equals $-1$ for any such $q$. When $a'>\sqrt{U'}$, an additional shortcut is available by noting that for $n'=k^2\leq U'$, the value $k$ belongs to the set of square roots of $-c'/b'$ modulo $a'$ lifted to the interval $[1,a')$. Such square roots can be routinely computed and allow us to obtain all solutions $n'\leq U'$ at once.

In some cases, we can also recognize the oddness of $\sigma(n')$ regardless of the parity of $n'$, for example, when $a'$ and $c'$ are odd while $b'$ is even. In this case, $n'$ can be a square or twice a square.
Accordingly, we extend the test described above to $p=2$ by computing the Legendre symbol $\left(\frac{-2^tb'c'}{q}\right)=\left(\frac{-2^{t\bmod 2}b'c'}{q}\right)$. In particular, this test automatically eliminates the possibility of even quasiperfect solutions satisfying $\sigma(n)=2n+1$ (see Section~\ref{sec:small_abundance}) since for any exponent $t\geq 1$, $\sigma(2^t)=2^{t+1}-1$ has a prime factor $q$ congruent to 3 modulo 4, giving the Legendre symbol $\left(\frac{-2^tb'c'}{q}\right) = \left(\frac{-2^{t+1}}{q}\right) = \left(\frac{-1}{q}\right) = -1$.

We also recognize the squareness of $n'$ when a known value of $\tau(n')$ (see Section~\ref{sec:add_constraints}) is odd.

\subsection{Case of $\gcd(a',c')>1$}

From $\gcd(a',b',c')=1$, it follows that $g:=\gcd(a',c')$ divides any solution $n'$. Suppose that $g>1$.
If $\gcd(g,m)>1$, then there are no solutions as $n'$ is coprime to $m$. However, if $\gcd(g,m)=1$, the prime factors of $g$ give forced prime factors of $n'$. In this case, instead of rolling the wheel in search for $\spf(n')$, we choose the largest prime power $p^e$ from the prime factorization of $g$ and define 
$Q = \{ p^t\ :\ t = e, e+1, \dots, e+\lfloor \log_p\frac{U'}{g}\rfloor\}$.
Jumping from $m$ to a node $m':=mq$ for $q\in Q$ yields a narrower search for $n'$.

Since solutions of the form $n=m'n''$ need no longer satisfy the restriction $\spf(n'')>\lpf(m')$,
to properly incorporate such jumps into the search, we introduce and maintain a lower bound $l_p$ for $\spf(n')$ independent of $\lpf(m)$ (in particular, $l_p$ does not change when we jump from $m$ to $m'$). 
Also, to guarantee that $\gcd(m,n')=1$, we make the prime wheel roll over the set of primes excluding the prime factors of $m$.

\section{{\sc SageMath} implementation}\label{sec:implement}

\subsection{RES framework}

The described traversal of $T_U$ fits naturally into the framework of \emph{recursively enumerated sets} (RES) in the {\sc SageMath} computer algebra system~\cite{Sage}. It allows efficient traversal of the nodes of a forest (tree $T_U$ in our case) by specifying seeds (here, the root $1$ of $T_U$) and defining a function {\tt succ($t$)} that computes the set of successors of a given node $t$. For the implementation, we represent $t$ by the tuple 
$(a',b',c', m, l_p, \mathrm{aux})$, where the first five elements have the same meaning as in the previous section, and $\mathrm{aux}$ is a dictionary specifying additional constraints (see Section~\ref{sec:add_constraints} below). Thus, the tuple $t$ may be viewed as the \emph{configuration} of node $m$ in $T_U$.

\subsection{Configuration reduction}\label{sec:reduceabc}

To simplify configuration handling, we define a local function {\tt reduce\_abc($t$)}, which reduces the given configuration $t$ (for example, by canceling the common factor of $a',b',c'$) and returns the resulting reduced configuration. It also recognizes configurations $t$ with no solutions and returns {\tt None}, indicating that traversal of the subtree rooted at $t$ should be avoided. For example, if $\gcd(a',b',c')=1$ but $\gcd(a',c')$ has a prime factor below $l_p$, then that prime would have to divide $n'$, contradicting the lower bound on its prime factors. 

Another special case recognized by {\tt reduce\_abc} is $a'=b'=c'$, where every prime $p$ gives a solution. However, in view of the given $m$ and $l_p$, such primes must satisfy $p\geq l_p$ and $p\nmid m$. The function {\tt reduce\_abc($t$)} prints a message describing the corresponding infinite family of solutions, and returns {\tt None} rather than traversing the corresponding subtree. We show an example of an equation with an infinite family of solutions in Section~\ref{sec:small_abundance} below.

Because some equations of the form~\eqref{eq:main} have already been studied extensively by computation, our implementation can optionally recognize such ``core'' equations (parameter {\tt refs}) and refer to the corresponding OEIS sequences. With {\tt refs=True}, once a configuration $t$ is identified as corresponding to a core equation, a message with a reference to the corresponding OEIS sequence is printed and no processing of $t$ takes place. In particular, equations $(a',b')=(1,2)$ and small even $c'$ (discussed in Section~\ref{sec:small_abundance}) can be handled this way because their solutions below $10^{20}$ can be queried from the OEIS.

\subsection{{\sc MapReduce} parallelization}

The primary benefit of the RES framework is readily available parallelization via the {\sc MapReduce} mechanism~\cite{Hivert2017} present in {\sc SageMath}. In addition to parallelizing the traversal, it supports parallel processing of each visited node $t$ via a user-defined function {\tt proc($t$)}, which computes a result (for example, a set of solutions) for that node; the results can then be combined over all visited nodes. 
In our case, the prime wheel (which computes successors) is implemented in {\tt succ($t$)}, whereas the shortcuts (which produce actual solutions) are implemented in {\tt proc($t$)}.
The wheel is local to a configuration: different configurations share no wheel state, and after successor generation their subtrees can be processed independently. Thus the mathematical algorithm introduces no global synchronization through the prime wheel; practical parallel scaling is governed by the scheduling backend and by the unequal sizes of the resulting subtrees.

\subsection{Additional constraints}\label{sec:add_constraints}

It is possible to further narrow the traversal by enforcing additional constraints.
Our implementation supports the following constraints via optional parameters: 
\begin{itemize}
    \item squarefreeness of $n$ (parameter {\tt squarefree});
\item evenness of $n$ (parameter {\tt even\_only});
\item coprimality to a given integer (parameter {\tt coprime\_to});
\item bounds on $\omega(n)$ and $\Omega(n)$ (parameters {\tt omega} and {\tt bigomega}, respectively);
\item a prescribed value for $\tau(n)$ (parameter {\tt numdiv}). 
\end{itemize}
Nontrivial constraints, whether they are derived from the given parameters or obtained while rolling the prime wheel in {\tt succ()} function, are passed from a parent node to its children in the {\tt aux} dictionary, thereby narrowing the descendant searches. Such constraints also reduce the work required by the shortcuts in {\tt proc()}: for example, a bound like $\omega(n')\geq 2$ implies that the case $n'=p^k$ is impossible and can be skipped, and similarly a bound like $\Omega(n')\geq 3$ implies that the case $n'=pq$ is impossible.

\subsection{Solutions above $U$}\label{sec:aboveU}

As noted above, the shortcuts described in Section~\ref{sec:shortcuts} can potentially produce some solutions above $U$. In our implementation, the parameter {\tt strict} controls whether such larger solutions are sought and reported, at the cost of moderate overhead. In our computational experiments, some of which are described in Section~\ref{sec:applications}, large solutions---whether previously known or newly discovered---can motivate us to increase the search bound and thereby certify their position among all solutions. 
Unfortunately, some of the discovered solutions, such as the larger of the two $2772$-hyperperfect numbers reported in Section~\ref{sec:hyperperfect}, are too large to be reached by a complete search at the corresponding bound.

\subsection{Availability}

Our implementation is available from the following GitHub repository:
\begin{center}
\url{https://github.com/maxale/multiplicative_functions}    
\end{center}

Our method is accessible via function \nolinkurl{res_solve_sigma_abc()} in the code file \nolinkurl{sigma_linear_eq.sage}. It requires the arguments $a$, $b$, $c$, and $U$, and also supports optional parameters, some of which are described above. The code documents the full list of supported parameters and their format.

\enlargethispage{4pt}
\section{Applications}\label{sec:applications}

In this section, we present some practical results obtained with our method for various equations of interest.

\subsection{Numbers with a small abundance}\label{sec:small_abundance}

The \emph{abundance} of a number $n$ is defined as $\sigma(n)-2n$. The perfect numbers have abundance $0$, so the abundance of $n$ can be viewed as the ``distance'' from $n$ to being a perfect number. 

Two notable cases are numbers of abundance $1$, called \emph{quasiperfect numbers}, and numbers of abundance $-1$, called \emph{almost-perfect numbers}. Whether quasiperfect numbers exist remains an open question. It is known that quasiperfect numbers must be odd squares greater than $10^{35}$~\cite{Hagis1982}. 
With our method, we raise this bound to $10^{90}$, which was established in about $2858$ core-hours.\footnote{We define \emph{core-hours} as the wall-clock time in hours multiplied by the number of cores used. Most experiments were run on Intel Xeon 2.40GHz or AMD EPYC 2.2GHz CPUs with up to 80 cores.} 
As explained in Section~\ref{sec:odd_sigma}, the squareness and oddness of possible solutions are automatically detected and taken into account by our method.

The only known almost-perfect numbers are the powers of $2$. The existing literature on almost-perfect numbers does not appear to provide an explicit lower bound on almost-perfect non-powers of $2$, but focuses on the possible structure of such numbers (for example, see \cite{Kishore1981}).
With our method, we establish that no other almost-perfect numbers exist below $10^{45}$, requiring about $1195$ core-hours. For the odd almost-perfect numbers other than $1$, we establish that none exist below $10^{90}$, requiring about $826$ core-hours.

The exceptionally large $10^{90}$ bounds for quasiperfect and odd almost-perfect numbers rely on the additional structure available when $\sigma(n')$ is odd. As described in Section~\ref{sec:odd_sigma}, the residual is then restricted to a square or twice a square, and candidate prime powers are further pruned by quadratic-residue tests using Legendre symbols. These computations should therefore not be viewed as representative of the general equation; Table~\ref{tab:fixed_abundance} gives examples of more typical search bounds when this odd-$\sigma$ structure is unavailable.

In general, numbers with an odd abundance are much sparser than those of even abundance, since an odd abundance of $n$ implies the oddness of $\sigma(n)$, and thus $n$ must be a square or twice a square.
The Online Encyclopedia of Integer Sequences~\cite{OEIS} contains sequences for every even abundance in the interval $[-32,32]$ as well as for abundances in $\{\pm 42, \pm 54, \pm 64, \pm 90, 128\}$.
With our method, we have routinely completed these sequences through $10^{20}$. For some of them we reached larger bounds, typically chosen to match some term discovered by the shortcuts (for example, a term produced by Theorem~\ref{th:abundance}).
In Table~\ref{tab:fixed_abundance},
we list some of the largest bounds we achieved and their corresponding running times in core-hours.

\begin{table}[!t]    
\begin{center}
\begin{tabular}{c|c|c|c}
Abundance & OEIS & Search bound & Core-hours \\
\hline
-2  & {\tt A191363} & $10^{28}$ & 1785 \\
-10 & {\tt A101223} &  $9.9\cdot 10^{27}$ & 5040 \\
12  & {\tt A141545} & $10^{27}$ & 235 \\
-22 & {\tt A223606} & $3.9\cdot 10^{28}$ & 1791 \\
24 & {\tt A223613} & $2.7\cdot 10^{26}$ & 1148 \\
-54  & {\tt A101259} & $3.8\cdot 10^{25}$ & 3132
\end{tabular}
\caption{Selected constant-abundance sequences in the OEIS, along with the achieved search bounds and the approximate search times.}
\label{tab:fixed_abundance}
\end{center}
\end{table}

The sequence of numbers of abundance $12$ contains the infinite subsequence $(6\p_k)_{k\geq 3}$ and thus the corresponding OEIS sequence {\tt A141545} consists mostly of small terms from this subsequence. Our method correctly identifies this infinite subsequence (by printing a message about it) and focuses the search on \emph{sporadic} solutions outside it. Those sporadic solutions form a subsequence of OEIS sequence {\tt A234238}, which lists sporadic solutions to the more general congruence $\sigma(n)\equiv 6\pmod{n}$. We completed the search for this broader congruence below $10^{27}$.

\subsection{Hyperperfect numbers}\label{sec:hyperperfect}

Hyperperfect numbers represent another generalization of perfect numbers~\cite[Section B2]{Guy2004}.
A positive integer $n$ is called \emph{$k$-hyperperfect}, for an integer $k$, if $n = 1 + k (\sigma(n)-n-1)$, where $\sigma(n)-n-1$ is the sum of the divisors of $n$ other than $1$ and $n$ itself. The $1$-hyperperfect numbers are precisely the perfect numbers. 
McCranie~\cite{McCranie2000} tabulated hyperperfect numbers below $10^{11}$ and identified a few values of $k$ of particular interest. For several values $k\ne 1$, the OEIS contains sequences of $k$-hyperperfect numbers, listed in Table~\ref{tab:hyperperfectOEIS}.

\begin{table}[!h]   
\begin{center}
\begin{tabular}{c||c|c|c|c|c|c|c}
k & 2 & 4 & 6 & 12 & 18 & 2772 & 31752 \\
\hline
OEIS & {\tt A007593} & {\tt A220290} & {\tt A028499} & {\tt A028500} & {\tt A028501} & {\tt A028502} & {\tt A034916}
\end{tabular}    
\caption{OEIS sequences of $k$-hyperperfect numbers (excluding perfect numbers).}
\label{tab:hyperperfectOEIS}
\end{center}
\end{table}

The defining equation for $k$-hyperperfect numbers has the form~\eqref{eq:main} with $(a,b,c)=(k, k+1, k-1)$, so we apply our method to determine all terms in the cited sequences below bounds of at least $10^{20}$. In addition to extending the search bounds and certifying the order of known terms, we discovered some previously unknown 
hyperperfect numbers, such as the following two $2772$-hyperperfect numbers having three and four prime factors, respectively:
$$
47268697363953913 = 2791\cdot 411409\cdot 41166127
$$
and
$$
186690534609915040044368953 = 5237\cdot 6173\cdot 128669\cdot 44881723181837.
$$
While the former number is below our search bound and is proved to be the fifth $2772$-hyperperfect number in order, the latter remains out of reach and thus its position in the sequence is unknown.

While $2$-hyperperfect numbers satisfy the equation $2\sigma(n)=3n+1$, the OEIS sequence {\tt A063906} lists solutions to a similar equation $2\sigma(n)=3n+3$, which can also be written as $\sigma(n)=\frac{3}2(n+1)$ in a form reminiscent of the perfect-number equation. We determined all solutions to this equation below $3.7\cdot 10^{23}$, requiring 6336 core-hours, as well as discovered some previously unknown terms above that bound. 

\subsection{$f$-perfect numbers}

For a given arithmetic function $f$, $f$-perfect numbers are defined~\cite{Pe2002} as integers $n$ satisfying $2f(n) = \sum_{d\mid n} f(d)$. For the identity function $f$, they are exactly the perfect numbers, and thus $f$-perfect numbers represent yet another generalization of the perfect numbers.
The OEIS contains a few sequences listing $f$-perfect numbers, including $f(n)=n+1$ (sequence {\tt A066229}) and $f(n)=n-1$ (sequence {\tt A066230}).

If $f$ is a linear function, say $f(n)=un+v$ with integer coefficients $u,v$, then the defining equation of $f$-perfect numbers becomes $2(un+v)=u\sigma(n) + v\tau(n)$. For a fixed value of $\tau(n)=d$, it takes the form~\eqref{eq:main} with $(a,b,c)=(u,2u,v(2-d))$, to which we can apply our method under the additional constraint $\tau(n)=d$ (see Section~\ref{sec:add_constraints}). We identify the feasible values of $\tau(n)$ as follows.

The bound $n\leq U$ implies an upper bound for $\tau(n)$. For $U<10^{480}$, an accurate bound can be obtained from data present in the OEIS sequence {\tt A002182} of \emph{highly composite numbers}, which are the numbers $k$ such that $\tau(k)>\tau(\ell)$ for all $\ell<k$. Namely, if $k$ is the largest highly composite number $\leq U$, then for any $n\leq U$, we have $\tau(n)\leq D:=\tau(k)$. We can then quickly identify feasible values of $d$ in the interval $[1,D]$ by checking if the smallest number $m$ with $\tau(m)=d$ (OEIS sequence {\tt A005179}) does not exceed $U$. 

Using this procedure, we determined all $(x+1)$-perfect numbers below
$10^{26}$, including the following newly discovered term with many prime factors:
$$
20055918935605248255 = 3\cdot 5\cdot 7^3\cdot 17\cdot 101\cdot 719\cdot 991\cdot 3186283.
$$
Similarly, we determined all $(x-1)$-perfect numbers below $10^{26}$, thereby determining the position of the term 
$$
98354828344320469458616065 = 3 \cdot 5 \cdot 7 \cdot 17 \cdot 37 \cdot 173 \cdot 56711 \cdot 240479 \cdot 631197361
$$
recently discovered by Alexander Violette.

\section{Complexity and search-space analysis}\label{sec:complexity}

The applications in Section~\ref{sec:applications} demonstrate both the practical reach of the algorithm and the strong dependence of its running time on the equation coefficients and on arithmetic constraints discovered along individual branches. A full bit-complexity bound must also account for integer factorization and primality testing in the shortcuts. We therefore analyze the prime-wheel search at a fixed configuration, then explain how its local complexity changes down the recursive tree, and finally relate this analysis to the observed computational behavior.

\subsection{Maximal prime wheels}\label{sec:maximal_wheels}
The prime-wheel part admits a useful local estimate. Consider one fixed configuration, relabel its local parameters as $(\alpha,\beta,\gamma,V)$, suppose $\gamma>0$, and for this estimate ignore prime powers, shortcuts, gcd-jumps, and additional constraints. By~\eqref{eq:LW},
\[
 B=\beta+\frac{\gamma}{V},\qquad \rho:=\frac{B}{\alpha}
\]
is independent of the wheel. For each possible starting prime $x$, let $\widehat W_x=(x,\ldots,y)$ be the longest prime wheel satisfying $P_0(\widehat W_x)\leq V$; equivalently, if $\widehat W_x^+$ denotes its length increment, then
\[
 P_0(\widehat W_x)\leq V<P_0(\widehat W_x^+).
\]
Both $P_0(W)$ and $P_0(W)/P_1(W)$ increase under a length increment. Hence there exists a product-feasible wheel starting at $x$ that passes the wheel-ratio test $\alpha P_0(W)/P_1(W)\geq B$ if and only if the maximal product-feasible wheel $\widehat W_x$ passes it. The wheel $\widehat W_x$ is only an analysis device: the implementation stops increasing the length as soon as the ratio test succeeds and then rolls the wheel.

\begin{thm}\label{thm:wheel_count}
Fix $\rho>1$, and let $M_\rho(V)$ be the number of starting primes $x$ for which
\[
 \frac{P_0(\widehat W_x)}{P_1(\widehat W_x)}\geq \rho.
\]
Then
\[
 M_\rho(V)=(\log V)^{1/\rho+o(1)}\qquad (V\to\infty).
\]
\end{thm}
\begin{proof}
Put $T=\log V$ and write $y$ for the last prime of $\widehat W_x$. With $\vartheta(z):=\sum_{p\leq z}\log p$, maximality gives
\[
 \vartheta(y)-\vartheta(x^-)=T+O(\log y).
\]
The prime number theorem gives $\vartheta(z)\sim z$, while Mertens' product theorem gives, whenever $x\to\infty$,
\[
 \frac{P_0(\widehat W_x)}{P_1(\widehat W_x)}
 =\prod_{x\leq p\leq y}\frac{p}{p-1}
 =\frac{\log y}{\log x}\,(1+o(1)).
\]
Fix $0<\varepsilon<1-1/\rho$. Uniformly for primes $x\leq T^{1/\rho-\varepsilon}$ that tend to infinity with $V$, we have $x=o(T)$, so maximality implies $y=T^{1+o(1)}$ and the displayed ratio is at least $(1/\rho-\varepsilon)^{-1}+o(1)>\rho$; the finitely many bounded starting primes are admissible for all sufficiently large $V$ as well. If $T^{1/\rho+\varepsilon}\leq x\leq T$, maximality and the prime number theorem give $y=O(T)$, so the ratio is at most $(1/\rho+\varepsilon)^{-1}+o(1)<\rho$. Finally, if $x>T$, then a wheel with product at most $V$ has at most $T/\log x$ primes and therefore
\[
 \frac{P_0(W)}{P_1(W)}
 \leq \left(\frac{x}{x-1}\right)^{T/\log x}=1+o(1)<\rho.
\]
Consequently, for every such $\varepsilon$ and all sufficiently large $V$, the admissible primes contain all primes up to $T^{1/\rho-\varepsilon}$ and none above $T^{1/\rho+\varepsilon}$. The prime number theorem for $\pi(x)$ then yields $M_\rho(V)=T^{1/\rho+o(1)}$.
\end{proof}

\begin{cor}\label{cor:wheel_root}
At the root configuration, fix $a>0$, $b>a$, and an integer $c$. In the simplified setting of this subsection, let $M(U)$ be the number of starting primes for which some product-feasible prime wheel passes the actual test $aP_0(W)/P_1(W)\geq L(W)$, with $L(W)$ defined by~\eqref{eq:LW}. Then $M(U)=(\log U)^{a/b+o(1)}$. In particular, for $\sigma(n)=2n+d$ with any fixed integer $d$, $M(U)=(\log U)^{1/2+o(1)}$.
\end{cor}
\begin{proof}
Put $\rho_0=b/a>1$. For $c\geq0$, $L(W)/a=\rho_0+O(U^{-1})$. For $c<0$, $L(W)/a=\rho_0+c/(aP_0(W))$; in the transition range of Theorem~\ref{thm:wheel_count}, $P_0(W)\geq x\to\infty$, so this is $\rho_0+o(1)$ uniformly. The same $\varepsilon$-sandwich as in that theorem, using fixed thresholds arbitrarily close to $\rho_0$, gives $M(U)=(\log U)^{1/\rho_0+o(1)}=(\log U)^{a/b+o(1)}$; bounded starting primes do not affect the asymptotic count.
\end{proof}

The same argument gives a useful finite-range approximation when the wheel-ratio threshold is constant. Using the approximations $\vartheta(z)\approx z$ and $\prod_{p\le z}p/(p-1)\approx e^{\gamma_E}\log z$, where $\gamma_E$ is Euler's constant, gives $y-x\approx T$ and $y\approx x^\rho$. Thus the last relevant starting prime is predicted by the positive solution $X=X_\rho(T)$ of
\begin{equation}\label{eq:wheel_threshold}
 X^\rho-X=T=\log V,
\end{equation}
and hence $M_\rho(V)\approx \pi(X)\sim \rho(\log V)^{1/\rho}/\log\log V$. For $\sigma(n)=2n+d$ with $d>0$, the root has the constant threshold $\rho=2+d/U$; at $U=10^{28}$, for example,~\eqref{eq:wheel_threshold} gives $X\approx8.5$, corresponding to starting primes $2,3,5,7$. For $d<0$, Corollary~\ref{cor:wheel_root} gives the same asymptotic exponent, but the finite-$U$ threshold $2+d/P_0(W)$ depends on the wheel, so~\eqref{eq:wheel_threshold} is not a precise finite-range cutoff in that case.

\subsection{Dependence on the configuration}\label{sec:complexity_config}
This estimate is local rather than a global complexity bound, because the effective ratio changes down a recursive branch. Returning to the original parameters $(a,b,c,U)$, at the node corresponding to a partial product $m$ the uncancelled local parameters are $(a\sigma(m),bm,c,U/m)$. Common-factor cancellation does not affect their ratio, and therefore
\begin{equation}\label{eq:rho_m}
 \rho_m=\frac{m}{\sigma(m)}\,\frac{b+c/U}{a}.
\end{equation}
Whenever $\rho_m>1$, Theorem~\ref{thm:wheel_count} predicts $(\log(U/m))^{1/\rho_m+o(1)}$ candidate starting primes at that configuration. Since $m/\sigma(m)<1$ for $m>1$, $\rho_m$ can decrease substantially down the tree; as it approaches $1$, the wheel becomes less selective, and the theorem does not yield a power saving in $\log U$ once $\rho_m\leq1$. This supplies a theoretical mechanism for the strong dependence of running time on both the equation coefficients and the branch being explored. Summing these local costs over all recursively generated configurations, together with the arithmetic cost of shortcuts, remains the main obstacle to a sharp global bound.

Every recursive descent multiplies the current partial product $m$ by an integer greater than $1$, so the depth of a search path is at most $\lfloor\log_2 U\rfloor$. A streaming serial depth-first traversal therefore needs only $O(\log U)$ simultaneously active configurations. Our present SageMath/RES implementation additionally materializes successor lists and, under parallel execution, frontier states; its actual memory use also includes those frontiers, prime-generation data, and the workspaces of the arithmetic subroutines.

\subsection{Empirical behavior and comparison with earlier algorithms}\label{sec:complexity_empirical}
The organization of the tree $T_U$ is closely related to classical recursions that split integers according to an extremal prime factor, notably Buchstab's identity; Bach and Sorenson~\cite{BachSorenson2020} use such a recursion to enumerate factored smooth integers. It is also an instance of reverse search~\cite{Avis1996} and is related to the search method of Fang~\cite{Fang2022}. Earlier work connects evaluation of arithmetic functions with factorization: Bach, Miller, and Shallit~\cite{BachMillerShallit1986} treat $\sigma$, while Shallit and Shamir~\cite{ShallitShamir1985} and Bach and Shallit~\cite{BachShallit1996} give broader algorithmic context.

A natural baseline is to sieve all integers up to $U$, compute $\sigma(n)$ from their factorizations, and test~\eqref{eq:main}. A smallest-prime-factor sieve uses $O(U\log\log U)$ marking operations and $\Theta(U\log U)$ bits when all smallest prime factors are stored; Helfgott~\cite{Helfgott2020} gives an Eratosthenes-type sieve with $O(U\log U)$ time and $O(U^{1/3}(\log U)^{2/3})$ space. Our method avoids scanning all $n\leq U$: the shortcuts handle a prime power or a product of two distinct primes directly, while Theorem~\ref{th:prime_wheel} guarantees that prime-wheel pruning cannot discard a feasible continuation. Thus the algorithmic contribution lies in equation-specific, completeness-preserving pruning together with these shortcuts.

The wheel estimate above grows as a power of $\log U$ at any fixed configuration with $\rho>1$, whereas a sieve performs work tied directly to $U$ even when solutions are extremely sparse. For general constant-abundance equations, Table~\ref{tab:fixed_abundance} gives representative search bounds between about $10^{25}$ and $10^{28}$. The much larger $10^{90}$ bounds for quasiperfect and odd almost-perfect numbers arise in the special odd-$\sigma$ regime of Section~\ref{sec:odd_sigma}, where square/twice-square restrictions and quadratic-residue tests on candidate prime powers provide substantially stronger pruning. These exceptional bounds should therefore not be read as a coefficient-independent complexity claim; they illustrate how additional arithmetic structure can further shrink the same search framework.

Empirically, for many tested equations the running time is well approximated by $r^{\log_{10}U}=U^{\log_{10}r}$ where the equation-dependent parameter $r$ is commonly between $2$ and $4$, corresponding to empirical exponents of about $0.30$--$0.60$.

Our computations also tend to scale well with the number of cores; for example, increasing from 40 to 80 cores nearly halves the running time.

\section{Concluding remarks}\label{sec:remarks}

We devoted substantial effort to high-level algorithm design---primarily minimizing the number of nodes of $T_U$ that must be visited---and to fitting the method into the RES/{\sc MapReduce} framework, but have done relatively little low-level optimization. Since {\sc SageMath} is {\sc Python}-based, the implementation incurs interpreter overhead. We expect that \emph{cythonization} of our implementation or re-implementing it in a parallelization-aware lower-level language (such as {\sc C++} with {\sc Cilk}) could yield substantial speedups. We plan to explore this in future work.

Another avenue for scaling the method is distributed parallelism across multiple computers. We believe it is well suited to distribution across multiple nodes of a computational cluster as well as across heterogeneous computers in a crowd-computing project similar to~\cite{Fang2018}, although we have not pursued this in practice.

One practical limitation of the present method is that a completed search up to a bound $U_0$ cannot in general be extended incrementally to a larger bound $U_1>U_0$. Increasing the bound $U$ changes $U'=U/m$ throughout the shortcut and prime-wheel tests, so a branch legitimately pruned for $U_0$ may become feasible for $U_1$; consequently, the search must be regenerated (although a known lower bound can still be used in some pruning inequalities, such as~\eqref{eq:LW} when $c'<0$). This is distinct from resuming an interrupted computation at a fixed $U$: such a traversal can in principle be checkpointed by serializing the unexplored configurations $t=(a',b',c',m,l_p,\mathrm{aux})$ and continuing them later.

We also used Theorem~\ref{th:bound} and a similar computational approach in the collaborative effort~\cite{Tao2025}, which proved that the largest $n$ such that $L_n:=\mathrm{lcm}(1,2,\dots,n)$ is highly abundant is $n = 172$. In practice, our approach can determine whether $L_n$ is highly abundant for $n$ up to a few hundred (in particular, for all $n\leq 172$).

With a suitable adjustment of the shortcut and pruning techniques, our method can be used for linear equations with other multiplicative functions. In particular, we already have an efficient solver for linear equations with Euler's totient function; the manuscript describing it is currently in preparation.

\paragraph*{\bf Acknowledgments} The author thanks Jack Brennen, Wenjie Fang, Florent Hivert, Carl Pomerance, and Alexander Violette for thoughtful discussions of various aspects of this work and related topics.

\bibliographystyle{plainurl}
\bibliography{refs}

\end{document}